\documentclass[11pt]{amsart}

\usepackage[english]{babel}

\usepackage{amsmath}
\usepackage{graphicx}
\usepackage[colorlinks=true, allcolors=blue]{hyperref}
\usepackage{amsthm}
\usepackage{amssymb}
\usepackage{fullpage}

\usepackage{booktabs,tabularx}
\usepackage{amscd}
\usepackage{subfig}
\usepackage{xspace}
\usepackage{tikz}
\usetikzlibrary{positioning,automata,arrows,shapes,matrix, arrows, decorations.pathmorphing}
\tikzset{main node/.style={circle,draw,minimum size=0.3em,inner
sep=0.5pt}}
\tikzset{minu node/.style={circle,draw,minimum size=0.3em,fill=black!100,inner
sep=0.5pt}}
\tikzset{red node/.style={circle,minimum size=0.3em,draw=red!100,inner
sep=0.5pt}}
\tikzset{state node/.style={circle,draw,minimum size=2em,fill=blue!20,inner
sep=0pt}}
\tikzset{small node/.style={circle,draw,minimum size=0.5em,inner
sep=2pt,font=\sffamily\bfseries}}

\theoremstyle{plain}
\newtheorem{theorem}{Theorem}[section]
\newtheorem{lemma}[theorem]{Lemma}
\newtheorem{proposition}[theorem]{Proposition}

\newtheorem{corollary}[theorem]{Corollary}
\theoremstyle{definition}
\newtheorem{definition}[theorem]{Definition}
\newtheorem{example}[theorem]{Example}

\theoremstyle{remark}
\newtheorem{remark}[theorem]{Remark}
\numberwithin{equation}{section}

\newcommand{\R}{{\mathbb R}}
\newcommand{\Z}{{\mathbb Z}}
\newcommand{\C}{{\mathbb C}}
\newcommand{\ra}{\rightarrow}
\newcommand{\bfa}{\mathbf{a}}

\newcommand{\fg}{{\mathfrak g}}
\newcommand{\fk}{{\mathfrak k}}

\newcommand{\N}{\mathbb{N}}

\newcommand{\cala}{\mathcal{A}}
\newcommand{\calc}{\mathcal{C}}
\newcommand{\cald}{\mathcal{D}}

\newcommand{\bfz}{\mathbf{z}}
\newcommand{\bfx}{\mathbf{x}}
\newcommand{\bfy}{\mathbf{y}}
\newcommand{\akp}{\mathcal{A}_k(P)}
\newcommand{\lxp}{L(X,\omega_p)}
\newcommand{\hxp}{H(X,\omega_p)}
\newcommand{\pxp}{P(X,\omega_p)}
\newcommand{\dkab}{\mathcal{D}_k([a]\times[b])}
\newcommand{\res}{L\!\downarrow_\fk}
\newcommand{\se}{\subseteq}

\newcommand{\inverse}{^{-1}}

\newcommand{\abs}[1]{\lvert #1 \rvert}

\DeclareMathOperator{\width}{\mathrm{width}}

\title{Branching rules of minuscule representations via a new partial order}

\author{R.M. Green}
\address{Department of Mathematics, University of Colorado Boulder, Campus Box
395, Boulder, Colorado, USA, 80309}
\email{rmg@colorado.edu}

\author{Tianyuan Xu}
\address{Department of Mathematics and Statistics, University of Richmond,
Richmond, Virginia, USA, 23173}
\email{tianyuan.xu@richmond.edu}


\begin{document}
\maketitle

\begin{abstract}
  We introduce a new partial order on the set of all antichains of a fixed size in
any poset.  When applied to minuscule posets, these partial orders give rise to
distributive lattices that appear in the branching rules for minuscule
representations of complex simple Lie algebras.
\end{abstract}


\section{Introduction}

This paper introduces a new partial order $\le_k$ on the set $\cala_k(P)$ of all
antichains of a fixed size $k$ in any poset $P$. We assume basic familiarity
with poset theory, including the notions of antichains, order ideals, order
filters, covering relations, Hasse diagrams, products of posets, and
distributive lattices.  These notions can all be found in \cite[\S 3]{EC1},
whose definitions and notations we will follow. In particular, we write $a
\lessdot_P b$  to indicate two elements $a,b$ are in a covering relation in a
poset $P$. We denote the sets of positive integers and nonnegative integers by
$\Z_+$ and $\N$, respectively. For each $n\in \N$, we write $[n]$ for the set
$\{1,2,\cdots, n\}$, viewed as a poset with the natural order. Unless otherwise
stated, all posets in the paper will be finite.

It is well known that for a (finite) poset $P$, there is a bijection between the
set $J(P)$ of ideals of $P$ and the set $\cala(P)$ of antichains of $P$, given
by associating an ideal with its set of maximal elements. The containment order
on $J(P)$ then induces a partial order on $\cala(P)$, which we will denote by
$\le_J$, and which we may restrict to the set $\cala_k(P)$ for each $k\in \N$.
It is a classic result of Dilworth \cite{dilworth60} that when $k$ is the
\emph{width} of $P$, defined as $\width(P)=\max\{\abs{A}:A\in \cala(P)\}$, the
set $\akp$ is a distributive lattice under the restriction of $\le_J$.

The new partial order $\le_k$ we introduce on $\akp$ is defined as the reflexive
transitive extension of the relation $\prec_k$, where we declare $A\prec_k B$
for $A,B\in \akp$ if $B=A\setminus\{a\}\cup\{b\}$ for elements $a,b\in P$ such
that $a\lessdot_P b$.  As we show in Section \ref{sec:new_order}, the order
$\le_k$ is coarser than the restriction of $\le_J$ in general, and $\akp$ might
not be a distributive lattice under the order $\le_k$ when $k=\width(P)$.

The order $\le_k$ has striking properties when applied to minuscule posets in
the sense of Proctor \cite{proctor84}. Recall that a finite dimensional
representation of a simple Lie algebra is called a minuscule representation if
the Weyl group of the Lie algebra acts transitively on the weights of the
representation. A minuscule poset is a poset $P$ for which the poset $J(P)$ is
isomorphic to the weight poset of a minuscule representation. Both minuscule
representations and minuscule posets have well-known and explicit
classifications. In particular, the minuscule posets are the posets
of the forms $[a]\times [b]$, $J([n]\times [2]), J^m([2]\times [2]),
J^2([2]\times [3])$, and $J^3([2]\times [3])$ where $a,b,n\in \Z_+$ and $m\in
\N$. The Hasse diagrams of these posets are shown in Figure \ref{fig:minuscule}. 

In Theorem \ref{thm:dist}, we determine the posets of form $\akp$ for all
minuscule posets $P$, and we show that they are distributive lattices in all
cases.  This is notable since, as mentioned earlier, $\akp$ does not have to be
a distributive lattice for a general poset $P$. More remarkably, the sets $\akp$
for minuscule posets $P$ retain intimate connections to minuscule
representations: it turns out that the order $\le_k$ can be used to determine
the branching rules of all minuscule representations up to diagram automorphisms
of Lie algebras; see Theorem \ref{thm:repthy} and Remark \ref{rmk:automorphism}.
For example, for the minuscule poset $P = [a] \times [b]$, the poset $J(P)$ is
isomorphic to the weight poset of a suitable minuscule representation $V$ for a
simple Lie algebra of type $A$, and $\akp$ controls the branching of $V$. In
this case, the poset $\cala_k(P)$ is naturally isomorphic to the poset $\dkab$
of Young diagrams of Durfee length $k$ that fit into an $a \times b$ box; see
Definition \ref{def:Durfee} and Corollary \ref{cor:Durfee}. 

\begin{figure}[h!]
\centering
\subfloat[{$[a]\times [b]$.}]
{
\begin{tikzpicture} 
\node[main node] (0) {}; 
\node[main node] (1) [above left=0.4cm and 0.4cm of 0] {};
\node[main node] (2) [above left=0.4cm and 0.4cm of 1] {};
\node[main node] (3) [above left=0.4cm and 0.4cm of 2] {};

\node[main node] (4) [above right=0.4cm and 0.4cm of 0] {};
\node[main node] (5) [above right=0.4cm and 0.4cm of 4] {};
\node[main node] (6) [above right=0.4cm and 0.4cm of 5] {};
\node[main node] (7) [above right=0.4cm and 0.4cm of 6] {};

\node[main node] (40) [above right=0.4cm and 0.4cm of 1] {};
\node[main node] (50) [above right=0.4cm and 0.4cm of 40] {};
\node[main node] (60) [above right=0.4cm and 0.4cm of 50] {};
\node[main node] (70) [above right=0.4cm and 0.4cm of 60] {};

\node[main node] (400) [above right=0.4cm and 0.4cm of 2] {};
\node[main node] (500) [above right=0.4cm and 0.4cm of 400] {};
\node[main node] (600) [above right=0.4cm and 0.4cm of 500] {};
\node[main node] (700) [above right=0.4cm and 0.4cm of 600] {};

\node[main node] (4000) [above right=0.4cm and 0.4cm of 3] {};
\node[main node] (5000) [above right=0.4cm and 0.4cm of 4000] {};
\node[main node] (6000) [above right=0.4cm and 0.4cm of 5000] {};
\node[main node] (7000) [above right=0.4cm and 0.4cm of 6000] {};

\path[draw]
(0)--(1)
(2)--(3)
(3)--(4000)--(5000)
(2)--(400)--(500)
(1)--(40)--(50)
(0)--(4)--(5)
(6)--(7)
(60)--(70)
(600)--(700)
(6000)--(7000)
(4)--(40)
(5)--(50)
(6)--(60)
(7)--(70)
(400)--(4000)
(500)--(5000)
(600)--(6000)
(700)--(7000);

\path[draw,dashed]
(1)--(2)
(40)--(400)
(50)--(500)
(60)--(600)
(70)--(700)
(5)--(6)
(50)--(60)
(500)--(600)
(5000)--(6000);
\end{tikzpicture}
}
\qquad\qquad\qquad\qquad
\subfloat[{$J([n]\times [2])$}]{
\begin{tikzpicture} 
\node[main node] (0) {}; 
\node[main node] (1) [above right=0.4cm and 0.4cm of 0] {};
\node[main node] (2) [above right=0.4cm and 0.4cm of 1] {};
\node[main node] (3) [above right=0.4cm and 0.4cm of 2] {};
\node[main node] (4) [above right=0.4cm and 0.4cm of 3] {};

\node[main node] (10) [above left=0.4cm and 0.4cm of 1] {};
\node[main node] (20) [above right=0.4cm and 0.4cm of 10] {};
\node[main node] (30) [above right=0.4cm and 0.4cm of 20] {};
\node[main node] (40) [above right=0.4cm and 0.4cm of 30] {};

\node[main node] (100) [above left=0.4cm and 0.4cm of 20] {};
\node[main node] (200) [above right=0.4cm and 0.4cm of 100] {};
\node[main node] (300) [above right=0.4cm and 0.4cm of 200] {};

\node[main node] (1000) [above left=0.4cm and 0.4cm of 200] {};
\node[main node] (2000) [above right=0.4cm and 0.4cm of 1000] {};

\node[main node] (5) [above left=0.4cm and 0.4cm of 2000] {};

\path[draw]
(0)--(1)--(2)
(3)--(4)--(40)
(10)--(20)
(30)--(40)
(200)--(300)
(200)--(1000)--(2000)--(300)
(2000)--(5)
(10)--(1)
(20)--(2)
(30)--(3);

\path[draw,dashed]
(2)--(3)
(20)--(30)
(20)--(100)
(30)--(200)
(40)--(300)
(100)--(200);
\end{tikzpicture}
}\\
\subfloat[{$J^m([2]\times [2])$.}]
{
\begin{tikzpicture} 
\node[main node] (0) {}; 
\node[main node] (1) [above right=0.4cm and 0.4cm of 0] {};
\node[main node] (2) [above right=0.4cm and 0.4cm of 1] {};
\node[main node] (3) [above right=0.4cm and 0.4cm of 2] {};
\node[main node] (4) [above right=0.4cm and 0.4cm of 3] {};

\node[main node] (00) [above left=0.4cm and 0.4cm of 3] {};
\node[main node] (10) [above right=0.4cm and 0.4cm of 00] {};
\node[main node] (20) [above right=0.4cm and 0.4cm of 10] {};
\node[main node] (30) [above right=0.4cm and 0.4cm of 20] {};
\node[main node] (40) [above right=0.4cm and 0.4cm of 30] {};

\path[draw]
(0)--(1)
(2)--(3)--(4)--(10)--(00)--(3)
(10)--(20)
(30)--(40);

\path[draw,dashed]
(1)--(2)
(20)--(30);
\end{tikzpicture}
}
\qquad\qquad
\subfloat[{$J^2([2]\times [3])$.}]
{
\begin{tikzpicture} 
\node[main node] (0) {}; 
\node[main node] (1) [above right=0.4cm and 0.4cm of 0] {};
\node[main node] (2) [above right=0.4cm and 0.4cm of 1] {};
\node[main node] (3) [above right=0.4cm and 0.4cm of 2] {};
\node[main node] (4) [above right=0.4cm and 0.4cm of 3] {};

\node[main node] (00) [above left=0.4cm and 0.4cm of 2] {};
\node[main node] (10) [above right=0.4cm and 0.4cm of 00] {};
\node[main node] (20) [above right=0.4cm and 0.4cm of 10] {};

\node[main node] (01) [above left=0.4cm and 0.4cm of 10] {};
\node[main node] (02) [above right=0.4cm and 0.4cm of 01] {};
\node[main node] (03) [above right=0.4cm and 0.4cm of 02] {};

\node[main node] (11) [above left=0.4cm and 0.4cm of 01] {};
\node[main node] (12) [above right=0.4cm and 0.4cm of 11] {};
\node[main node] (13) [above right=0.4cm and 0.4cm of 12] {};
\node[main node] (14) [above right=0.4cm and 0.4cm of 13] {};
\node[main node] (15) [above right=0.4cm and 0.4cm of 14] {};

\path[draw]
(0)--(1)--(2)--(3)--(4)
(00)--(10)--(20)
(01)--(02)--(03)
(11)--(12)--(13)--(14)--(15)
(2)--(00)
(3)--(10)--(01)--(11)
(4)--(20)--(02)--(12)
(03)--(13);
\end{tikzpicture}
}
\qquad\qquad
\subfloat[{$J^3([2]\times [3])$.}]
{
\begin{tikzpicture} 
\node[main node] (0) {}; 
\node[main node] (1) [above right=0.4cm and 0.4cm of 0] {};
\node[main node] (2) [above right=0.4cm and 0.4cm of 1] {};
\node[main node] (3) [above right=0.4cm and 0.4cm of 2] {};
\node[main node] (4) [above right=0.4cm and 0.4cm of 3] {};
\node[main node] (5) [above right=0.4cm and 0.4cm of 4] {};

\node[main node] (00) [above left=0.4cm and 0.4cm of 3] {};
\node[main node] (10) [above right=0.4cm and 0.4cm of 00] {};
\node[main node] (20) [above right=0.4cm and 0.4cm of 10] {};

\node[main node] (01) [above left=0.4cm and 0.4cm of 10] {};
\node[main node] (02) [above right=0.4cm and 0.4cm of 01] {};
\node[main node] (03) [above right=0.4cm and 0.4cm of 02] {};

\node[main node] (11) [above left=0.4cm and 0.4cm of 01] {};
\node[main node] (12) [above right=0.4cm and 0.4cm of 11] {};
\node[main node] (13) [above right=0.4cm and 0.4cm of 12] {};
\node[main node] (14) [above right=0.4cm and 0.4cm of 13] {};
\node[main node] (15) [above right=0.4cm and 0.4cm of 14] {};

\node[main node] (21) [above left=0.4cm and 0.4cm of 11] {};
\node[main node] (22) [above right=0.4cm and 0.4cm of 21] {};
\node[main node] (23) [above right=0.4cm and 0.4cm of 22] {};
\node[main node] (24) [above right=0.4cm and 0.4cm of 23] {};
\node[main node] (25) [above right=0.4cm and 0.4cm of 24] {};

\node[main node] (31) [above left=0.4cm and 0.4cm of 24] {};
\node[main node] (32) [above right=0.4cm and 0.4cm of 31] {};
\node[main node] (33) [above right=0.4cm and 0.4cm of 32] {};
\node[main node] (34) [above right=0.4cm and 0.4cm of 33] {};
\node[main node] (35) [above right=0.4cm and 0.4cm of 34] {};

\path[draw]
(0)--(1)--(2)--(3)--(4)--(5)
(00)--(10)--(20)
(01)--(02)--(03)
(11)--(12)--(13)--(14)--(15)
(21)--(22)--(23)--(24)--(25)
(31)--(32)--(33)--(34)--(35)
(3)--(00)
(4)--(10)--(01)--(11)--(21)
(5)--(20)--(02)--(12)--(22)
(03)--(13)--(23)
(14)--(24)--(31)
(15)--(25)--(32);
\end{tikzpicture}
}
\caption{Minuscule posets.}
\label{fig:minuscule}
\end{figure}

Our motivation for studying the partial order $\le_k$ comes from our earlier work 
\cite{gx2} on Kazhdan--Lusztig cells of $\bfa$-value 2, where $\bfa$ is Lusztig's 
$\bfa$-function. For $\bfa(2)$-finite Coxeter groups (as defined and described 
in \cite{gx2}), every element $w$ of $\bfa$-value 2 has an associated
heap poset $H$. A key property of $H$ can be summarized as follows: the poset
$\cala_2(H)$ has a minimum element with respect to $\le_2$, and the ideal
generated by the minimal element determines the left cell of $w$. A similar
statement holds for maximal elements and right cells.

The rest of the paper is organized as follows. We introduce the order $\le_k$
and study its basic properties in Section \ref{sec:new_order}.  We introduce a
family of posets we call binomial posets in Section \ref{sec:sequence_posets}; 
the section does not treat the order $\le_k$ directly, but the binomial posets
will provide a useful model for the subsequent parts of the paper.  Sections
\ref{sec:type_A} and \ref{sec:other_types} study the poset $\cala_k(P)$ for
minuscule posets of type $A$ and of all other types, respectively, culminating
in the explicit descriptions of all such posets $\akp$ in Theorem
\ref{thm:dist}. Section \ref{sec:repthy} explains how the order $\le_k$ relates
to branching rules of minuscule representations of simple Lie algebras when
applied to minuscule posets. Finally, we discuss several open questions related
to the order $\le_k$ in Section \ref{sec:conclude}.

\section{A new partial order on antichains}
\label{sec:new_order}
Throughout this section, let $P$ be a poset, let $k\in \N$, and let $\cala_k(P)$
be the set of antichains of $P$ of cardinality $k$.  Recall from the
introduction that the containment order on ideals of $P$ induces an order
$\le_J$ on the antichains of $P$. The order $\leq_J$ can be
described without reference to ideals as follows: for any $A,B\in \cala(P)$, we
have  $$ A \leq_J B \quad \Longleftrightarrow \quad \forall\, a \in A, \exists\,
b \in B : a \leq_P b.$$ In this section, we will show that the partial order
$\le_k$ defined in the
introduction is indeed a partial order, study its basic
properties, and discuss its relationship with $\le_J$.
Recall that $\le_k$ is defined as follows. 

\begin{definition}\label{def:keydef}
For any $A,B\in \akp$, we write $A \prec_k B$ if $A \backslash B = \{a\}$ and $B
\backslash A = \{b\}$ are singleton sets with the property that $a <_P b$. We
define $\leq_k$ to be the reflexive transitive extension of $\prec_k$ on
$\cala_k(P)$.
\end{definition}

\begin{lemma}\label{lemma:new_poset}
The relation $\leq_k$ of Definition \ref{def:keydef} is a partial order on the
set $\cala_k(P)$, and the restriction of the partial order $\leq_J$ to
$\cala_k(P)$ refines $\leq_k$. \end{lemma}

\begin{proof}
The definition of $\prec_k$ implies that if $A, B \in \cala_k(P)$ satisfy $A \prec_k B$,
then we have $A \leq_J B$.  The antisymmetry of $\leq_J$ then implies the
antisymmetry of $\leq_k$, and it follows in turn that $\leq_k$ is a partial
order refined by the restriction of $\leq_J$ to $\akp$. 
\end{proof}

\begin{remark}
    \label{rmk:immediate} It is immediate from Definition \ref{def:keydef} that
      $\akp$ is nonempty if and only if $0\le k\le \width(P)$, that $\cala_0(P)$
      is the singleton poset, and that $\cala_1(P)$ is canonically isomorphic to
      $P$ itself.
  \end{remark}

\begin{proposition}\label{prop:basics}
Let $A,B\in \akp$. 
\begin{itemize} 
\item [(i)]  If $A \leq_k B$, then the elements of $A = \{a_1, a_2, \cdots,
     a_k\}$ and $B = \{b_1, b_2, \cdots, b_k\}$ can be ordered in such a way
     that $a_i \leq_P b_i$ for all $1 \leq i \leq k$.
 \item [(ii)] The elements $A$ and $B$ are in a covering relation
      $A\lessdot_{\cala_k(P)} B$ if and only if we have both (1) $A \prec_k B$
      and (2) the unique elements $a \in A \backslash B$ and $b \in B \backslash
      A$ satisfy $a \lessdot_P b$.
\end{itemize}
\end{proposition}

\begin{proof}
If $A \prec_k B$, then (i) follows by the definition of $\prec_k$; the general
case follows by induction.

To prove (ii), assume first that $A \lessdot_{\cala_k(P)} B$. By the definition
of $\leq_k$, we must have $A \prec_k B$. We then have $A = C \cup \{a\}$ and $B
= C \cup \{b\}$, where $C = A \cap B \in \cala_{k-1}(P)$ and $a, b \in P$
satisfy $a <_P b$. Suppose that some element $x \in P$ satisfies $a <_P x <_P
b$. Then the set $C' := C \cup \{x\}$ must be an antichain in  $\cala_k(P)$ for
the following reason: we cannot have $x \le_P c$ for any $c \in C$ because $a
<_P x \le_P c$ and $A$ is an antichain, and we cannot have $c \le_P x$ for any
$c \in C$ because $c \le_P x <_P b$ and $B$ is an antichain. It follows that $A
<_k C' <_k B$, which is a contradiction, so we have $a \lessdot_P b$.

Conversely, assume that $A, B \in \cala_k(P)$ satisfy $A \prec_k B$, and also
that $a \lessdot_P b$, where $C = A \cap B = \{c_1, c_2, \cdots, c_{k-1}\}$, $A
= C \cup \{a\}$, and $B = C \cup \{b\}$.  Write $A = \{a_1, a_2, \cdots, a_k\}$,
where $a_i = c_i$ for $i < k$ and $a_k = a$, and $B = \{b_1, b_2, \cdots,
b_k\}$, where $b_i = c_i$ for $i < k$ and $b_k = b$.  Suppose for a
contradiction that there exists $X = \{x_1, x_2, \cdots, x_k\} \in \cala_k(P)$
such that $A <_k X$ and $X <_k B$.  It follows from (i) that there are
permutations $\sigma$ and $\tau$ of $\{1, 2, \cdots, k\}$ such that for all $1
\leq i \leq k$, we have $a_i \leq_P x_{\sigma(i)}$ and $x_i \leq_P b_{\tau(i)}$,
which implies that $a_i \leq_P b_{\tau(\sigma(i))}$.  Because $A$ and $B$ are
antichains, we must have $\tau(\sigma(i)) = i$ for all $1 \leq i < k$, and this
implies that $\tau = \sigma^{-1}$. By relabelling $X$ if necessary, we may
assume that $\sigma$ and $\tau$ are both the identity permutation, and that $X =
C \cup \{x_k\}$, where $a = a_k <_P x_k <_P b_k = b$. This contradicts the
hypothesis that $a \lessdot_P b$, and (ii) follows.
\end{proof}

To compare $\le_k$ with $\le_J$, we first note that the order $\le_k$ may be
strictly coarser than the restriction of the order $\le_J$ to $A_k(P)$. For
example, consider the poset $P = \{a, b, c, d, e\}$  with covering relations $a
\lessdot_P c$, $b \lessdot_P c$, $c \lessdot_P d$, and $c \lessdot_P e$.  Then
$\width(P)=2$, and the set $\cala_2(P)$ consists of the two antichains, $\{a,
b\}$ and $\{d, e\}$. These antichains are comparable in the partial order
$\leq_J$, but not in the order $\le_k$ by Proposition \ref{prop:basics} (ii).
This shows that $\leq_2$ strictly coarsens $\leq_J$. The same example also shows
that the converse of Proposition \ref{prop:basics} (i) does not hold. On the
other hand, we note that for some important classes of examples, the partial
orders $\leq_J$ and $\leq_k$ on the maximal antichains of $P$ are identical.
Examples of such posets include the heaps of fully commutative elements (in the
sense of \cite{FC}) in finite Coxeter groups, and more generally in star
reducible Coxeter groups (in the sense of \cite{star_reducible}).

Dilworth \cite[Theorem 2.1]{dilworth60} proved that  $\cala_k(P)$ is a
distributive lattice under the partial order $\leq_J$ for $k=\width(P)$. The
poset $P$ from the previous paragraph proves that this is not the case for the
order $\leq_k$, since $\cala_2(P)=\cala_{\width(P)}(P)$ has no maximum or
minimum element under $\le_k$. Remarkably, however, all the nonempty posets
$\cala_k(P)$ are distributive lattices under the order $\leq_k$ if $P$ is a
minuscule poset, as we will prove in Theorem \ref{thm:dist}. Note that by Remark
\ref{rmk:immediate}, for $\cala_k(P)$ to be a distributive lattice for all
values of $k$, the poset $P$ must be a distributive lattice itself, but this is
not a sufficient condition: if $P=\{a,b,c\}$ is an antichain with three
elements, then $J(P)$ is a distributive lattice, but $\cala_2(J(P))$ is not
because it has no maximum or minimum element.

\section{Binomial posets}
\label{sec:sequence_posets}
In this section we study a family of posets, which we call binomial posets, that
will be used extensively in the descriptions of posets of the form $\akp$. We
show that binomial posets naturally parameterize posets of the form $J([a]\times
[b])$ (Proposition \ref{prop:jab_iso}) and a poset $\cald_k([a]\times[b])$ that
arises in the context of Young diagrams (Proposition \ref{prop:Durfee}).

\begin{definition}\label{def:cak}
For any $k,n\in \N$ such that $k \leq n$, we define $\calc(n,k)$ to be the poset
consisting of strictly increasing length-$k$ sequences with entries from $[n]$,
ordered by coordinatewise comparison: for sequences $\bfx = (x_1, x_2, \cdots,
x_k)$ and $\bfy = (y_1, y_2, \cdots, y_k)$ in $\calc(n, k)$, we define
$\bfx\le_{\calc(n,k)} \bfy$ if $x_i\le y_i$ for all $1\le i\le k$.  We
abbreviate the notation $\le_{\calc(n,k)}$ to $\le_\calc$ and
$\lessdot_{\calc(n,k)}$ to $\lessdot_\calc$.
\end{definition}

The fact that $\calc(n,k)$ is a poset is immediate from the above definition.
The name ``binomial poset'' is motivated by the fact that we may naturally
identify $\calc(n,k)$ with the set of all size-$k$ subsets of $[n]$, by
identifying each sequence in $\calc(n,k)$ with its set of entries.


\begin{lemma}\label{lem:cak} 
Let $k,n$ be nonnegative integers such that $k \leq n$. 
\begin{itemize}
\item[{\rm (i)}] 
The function $\rho : \calc(n, k) \rightarrow \N$ defined by 
\[
    \rho( (x_1,\cdots,x_k)) = x_1 + x_2 + \cdots + x_k
\]
is a rank function on $\calc(n,k)$.  In other words, if $\bfx, \bfy \in \calc(n,
k)$ satisfy $\bfx \leq_\calc \bfy$, then we have $\rho(\bfx) \leq \rho(\bfy)$,
with $\rho(\bfy) = \rho(\bfx) + 1$ if and only if $\bfx \lessdot_{\calc} \bfy$.

\item[{\rm (ii)}]
Two elements $\bfx=(x_1,\cdots,x_k)$ and $\bfy=(y_1,\cdots, y_k)$  in $\calc(n,
k)$  are in a covering relation $\bfx\lessdot_\calc \bfy$ if and only if there
exists $i\in [k]$ such that $y_i=x_i+1$ and $y_j=x_j$ for all $j\in
[k]\setminus\{i\}$.
\end{itemize}
\end{lemma}

\begin{proof}
The definition of $\leq_\calc$ implies that  $\rho(\bfx) < \rho(\bfy)$ whenever
$\bfx <_\calc \bfy$. It follows that if $\bfx <_\calc \bfy$ and $\rho(\bfy) =
\rho(\bfx) + 1$, then $\bfx \lessdot_\calc \bfy$.  To prove the converse,
suppose that $\bfx <_\calc \bfy$ for elements $\bfx =(x_1, x_2, \cdots, x_k)$
and $\bfy = (y_1, y_2, \cdots, y_k)$ in $\calc(n,k)$. Choose $1 \leq r \leq k$
to be the maximal index satisfying $x_r < y_r$, and let $\bfz = (\bfx \backslash
\{x_r\}) \cup \{x_r + 1\}$.  Note that the set $\bfz$ still consists of $k$
distinct numbers: if $r = k$ then $\bfz$ is obtained from $\bfx$ by increasing
the largest entry in $\bfx$, while if $r < k$ then we have $x_r + 1 \leq y_r <
y_{r+1}=x_{r+1}$.  Note also that regardless of whether $r = k$, we have $x
\leq_\calc z \leq_\calc y$ and $\rho(\bfz) = \rho(\bfx) + 1$. It follows that if
$\bfx \lessdot_\calc \bfy$ then we must have $\bfz = \bfy$ and $\rho(\bfy) =
\rho(\bfx) + 1$. This completes the proof of (i), and (ii) follows immediately.
\end{proof}

We now discuss two applications of binomial posets. The first application
concerns ideals of $[a]\times [b]$, and is illustrated in Example \ref{eg:seq}.
We think of $[a]\times [b]$ as embedded in the lattice $\Z_+^2 :=\{(x,y): x,y\in
\Z_+\}$, where a point $(i_1, j_1)$ is smaller than another point $(i_2,j_2)$ if
and only if $(i_1,j_1)$ lies weakly to the southwest of $(i_2,j_2)$. It follows that if $I$ is an ideal in $[a]\times [b]$ and
we define $m_j$ to the maximal integer $i\in [a]$ such that $(i,j)\in I$ for
each $j\in [b]$, then each $m_j$ records the number of elements in $I$ in Row
$j$ (i.e., in the set $\{(k,j):k\in \Z_+\}$), and the sequence 
\[ 
    \bfx_I=(m_b, \cdots, m_1) 
\] 
is a weakly increasing sequence such that $0\le m_b \le m_{b-1}\le \cdots\le
m_1\le a$.  The sequence 
\[ \bfx'_I=(m_b+1, m_{b-1}+2, \cdots, m_1+b) \] 
is then an element of the binomial poset $\calc(a+b,b)$.  Furthermore, it is
routine to verify that conversely every sequence in $\calc(a+b,b)$ has the form
$\bfx'_I$ for a unique ideal $I$ of $[a]\times [b]$, and that the map $f:
J([a]\times[b])\ra \calc(a+b,a), I\mapsto \bfx'_I$ satisfies the condition that
$I_1\se I_2$ if and only if $f(I_1)\le_\calc f(I_2)$ for all $I_1, I_2\in
J([a]\times[b])$. We have thus proved the following result (which holds
trivially if either $a$ or $b$ is zero).

\begin{proposition}\label{prop:jab_iso}
For any $a,b\in \N$, the posets  $J([a] \times [b])$ and $\calc(a+b, b)$ are
isomorphic.\qed 
\end{proposition}

\begin{example}
    \label{eg:seq}
If
$a=6, b=5,$ and $I$ consists of the filled vertices in the grid $[a]\times[b]$
shown in Figure \ref{fig:seq}, then
we have $x_I=(0,1,4,6,6)$ and $x'_I=(1,3,7,10,11)\in \calc(11,5)$.

\begin{figure}[h!]
            \centering
        \begin{tikzpicture}

\foreach \x in {0,...,5}
\foreach \y in {0,...,4}{
    \node[main node] at (\x,\y){ };
}

\foreach \x in {0,...,3}
\foreach \y in {0,...,2}{
    \node[minu node] at (\x,\y){ };
}

\foreach \x in {4,5}
\foreach \y in {0,1}{
    \node[minu node] at (\x,\y){ };
}

\node[minu node] at (0,3) {};

\foreach \x in {0,...,4}
\foreach \y in {0,...,3}{
    \draw (\x,\y+0.05) --(\x, \y+0.95);
    \draw (\x+0.05,\y) --(\x+0.95, \y);
    \draw (\x+0.05,4) --(\x+0.95, 4);
    \draw (5,\y+0.05) --(5, \y+0.95);
}

\foreach \x in {2,...,6}
\node[below,yshift=-1mm] at (\x-1,0) {\x};
\node[below,xshift=-3mm,yshift=-1mm] at (0,0) {1};

\foreach \y in {2,...,5}
\node[xshift=-3mm] at (0,\y-1) {\y};
\end{tikzpicture}
\caption{An ideal in the poset $[6]\times [5]\subset \Z_+^2$.}
\label{fig:seq}
\end{figure}
\end{example}

\begin{remark}
\label{rmk:symmetry}
\begin{itemize}
\item[{\rm (i)}]
Since $J(P)$ is a distributive lattice for any finite poset $P$, Proposition
\ref{prop:jab_iso} implies that binomial posets are distributive
lattices.
\item[\rm (ii)] 
Given two posets $P$ and $Q$, the map $P\times Q\ra Q\times P, (p,q)\mapsto
(q,p)$ is clearly a poset isomorphism, so Proposition \ref{prop:jab_iso} also
implies that $\calc(a+b,b)\cong J([a]\times [b])\cong J([b]\times[a]) \cong
\calc(a+b,a)$ as posets for all $a,b\in \N$. 
\end{itemize}
\end{remark}


The second application of binomial posets concerns Young diagrams or,
equivalently, Ferrers diagrams (see \cite[\S 1.7]{EC1}). Using the French
notation, we may conveniently view Ferrers diagrams to be finite ideals of the
infinite poset $\Z_+^2$.  Our goal is to use binomial posets to parameterize the
posets $\cald_k(a,b)$ defined below. 

\begin{definition}
\label{def:Durfee}
\begin{itemize}
\item[{\rm (i)}]
We define the \emph{Durfee length} of a Ferrers diagram $D$ to be the largest
integer $k\in \N$ such that $[k]\times[k]\se D$, i.e., the side length of the
largest square grid $S$ that fits inside $D$, and we call $S$ the \emph{Durfee
square} of $D$. 
\item[{\rm (ii)}]
For any $a,b,k\in \N$ with $k\le \min(a,b)$, we define $\cald_k(a,b)$ to be the
poset of all Ferrers diagrams with Durfee length $k$ contained in the set
$[a]\times[b]$, ordered by set containment. 
\end{itemize}
\end{definition}

Given any Ferrers diagram $D$ in the poset $\cald_k([a]\times[b])$, we may
naturally decompose $D$ into three parts: the Durfee square $S=[k]\times[k]$ of
$D$; the part $I_1=\{(i,j)\in D: j>k\}$  above $S$; and the part $I_2=\{(i,j)\in
D: i>k\}$ to the right of $S$. Shifting $I_1$ down and $I_2$ to the left by $k$,
we obtain ideals $I'_1=\{(i,j-k):(i,j)\in I_1\}$ and $I'_2=\{(i-k,j):(i,j)\in
I_2\}$ in the grids $[k]\times [b-k]$ and $[a-k]\times [k]$, respectively; see
Example \ref{eg:ferrers}.  It follows that we have a map \[\varphi: \dkab\ra
J([k]\times [b-k])\times J([a-k]\times [k]), \quad D\mapsto (I'_1,I'_2).\]
Furthermore, the map $\varphi$ is invertible, with the inverse $\varphi\inverse$
being the map that stacks $I'_1$ and $I'_2$ on top and to the right of the
Durfee square.  Both $\varphi$ and $\varphi\inverse$ clearly respect the orders
in $\cald_k([a]\times[b])$ and $\calc(a,k)\times\calc(b,k)$, so we have proved
the following:

\begin{proposition}
\label{prop:Durfee}
If $a,b,k\in \N$ satisfy $k\le \min(a,b)$, then the posets $\dkab$ and $
\calc(a,k)\times \calc(b,k)$ are isomorphic.\qed
\end{proposition}

\begin{example}
    \label{eg:ferrers}
If we view the ideal $I$ from Example \ref{eg:seq} as a Ferrers diagram
contained in the grid $[a]\times[b]$ with $a=6$ and $b=5$, then the Durfee length
of the Ferrers diagram is $k=3$, and the Durfee square is the
$[k]\times[k]=[3]\times[3]$ square in the bottom left corner. The natural
decomposition of the Ferrers diagram is illustrated in Figure
\ref{fig:ferrers}: the set $I_1$ in the decomposition consists of the
unique element in $I$ above the Durfee square, which forms an ideal of the
$[k]\times[b-k]=[3]\times[2]$ grid above $S$, and $I_2$ consists of the 7
elements in $I$ to the right of $S$, which form an ideal of the
$[a-k]\times[k]=[3]\times[3]$ grid to the right of $S$.

\begin{figure}[h!]
            \centering
        \begin{tikzpicture}

\foreach \x in {0,...,5}
\foreach \y in {0,...,2}{
    \node[main node] at (\x,\y){ };
}

\foreach \x in {0,...,2}
\foreach \y in {3,4}{
    \node[main node] at (\x,\y){ };
}

\foreach \x in {3,...,5}
\foreach \y in {0,...,2}{
    \node[main node] at (\x,\y){ };
}
\foreach \x in {0,...,3}
\foreach \y in {0,...,2}{
    \node[minu node] at (\x,\y){ };
}

\foreach \x in {4,5}
\foreach \y in {0,1}{
    \node[minu node] at (\x,\y){ };
}

\node[minu node] at (0,3) {};
\node[minu node] at (3,2) {};

\foreach \x in {0,1}
\foreach \y in {0,1,3}{
    \draw (\x,\y+0.05) --(\x, \y+0.95);
    \draw (\x+0.05,\y) --(\x+0.95, \y);
    \draw (\x+0.05,4) --(\x+0.95, 4);
}

\foreach \y in {0,1,3}{
\draw (2,\y+.05)--(2,\y+.95);
}

\draw (0.05,2)--(0.95,2);
\draw (1.05,2)--(1.95,2);

\foreach \y in {0,1}{
\draw (5,\y+0.05) --(5, \y+0.95);
}

\foreach \x in {3,4}{
\draw (\x+0.05,0)--(\x+0.95,0);
\draw (\x+0.05,1)--(\x+0.95,1);
\draw (\x+0.05,2)--(\x+0.95,2);
\draw (\x,0.05)--(\x,0.95);
\draw (\x,1.05)--(\x,1.95);
}


\end{tikzpicture}
\caption{Decomposition of the ideal $I$ from Example \ref{eg:seq}.}
\label{fig:ferrers}
\end{figure}
\end{example}

\section{Minuscule posets of type \texorpdfstring{$A$}{A}}\label{sec:type_A}


The goal of this and the next section is to describe the posets of the form
$\akp$ where $P$ is a minuscule poset. We start with the posets
$P=[a]\times[b]$, which correspond to minuscule representations of type $A$ as
we will explain in Section \ref{sec:repthy}. Viewing $P$ as embedded in the
lattice $\Z_+^2$ as in Section \ref{sec:sequence_posets}, we note that two
elements $(i_1,j_1)$ and $(i_2, j_2)$ form an antichain in $[a]\times [b]$ if and
only if either (a) $i_1 < i_2$ and $j_1 > j_2$, or (b) $i_1 > i_2$ and $j_1 <
j_2$. A basic induction then yields the
following characterization of antichains in $P$, which we record for ease of
reference.

\begin{lemma}
\label{lem:jab_antichain}
Let $P=[a]\times[b]$. Then a subset of $P$ forms an antichain in $P$ if and
only if it can be written in the form $A=\{(x_1,y_1), \cdots, (x_k,y_k)\}$ where
$x_1 < \cdots < x_k$ and $y_1 > \cdots > y_k$.\qed
\end{lemma}

The main result of the section is the following proposition, which describes
$\akp$ as a product of two binomial posets.

\begin{proposition}\label{prop:akab} 
If $a, b, k \in \N$ satisfy $k \leq \min(a, b)$, then 
$\cala_k([a]\times[b]) \cong \calc(a, k) \times \calc(b, k)$ as posets.
\end{proposition}

\begin{proof}
By Lemma \ref{lem:jab_antichain}, any antichain of $P$ of size $k$ can be
written uniquely as $$ A = \{ (x_1, y_1), \ (x_2, y_2), \ \cdots, \ (x_k, y_k)\}
,$$ where $1 \leq x_1 < x_2 < \cdots < x_k \leq a$ and $b \geq y_1 > y_2 >
\cdots > y_k \geq 1$. It follows that there is a function $\phi : \cala_k(P) \ra
\calc(a, k) \times \calc(b, k)$ given by $$ \phi(A) = (\phi_1(A), \phi_2(A)) =
((x_1, x_2, \cdots, x_k), \ (y_k, y_{k-1}, \cdots, y_1)) .$$ Furthermore, $\phi$
is a bijection, because the assignment \[((x_1, x_2,\cdots, x_k), (y_k,
y_{k-1},\cdots, y_1) )\mapsto \{(x_1,y_1), \cdots, (x_k,y_k)\}\] takes
$\calc(a,k)\times \calc(b,k)$ to $\akp$ by Lemma \ref{lem:jab_antichain} and is
clearly a two-sided inverse of $\phi$.

We claim that $\phi$ is an isomorphism of posets. To see this, it is enough to
prove that $\phi$ respects covering relations. Let $A \in \cala_k(P)$, and write
$A = \{ (x_1, y_1), \ (x_2, y_2), \ \cdots, \ (x_k, y_k)\}$, where $x_1 < x_2 <
\cdots < x_k$ and $y_1 > y_2 > \cdots > y_k$.

Suppose that we have $A' \lessdot_{\cala_k(P)} A$ for some $A'\in \akp$.
Proposition \ref{prop:basics} (ii) implies that $A'$ can be obtained from $A$ by
replacing one of the $(x_i, y_i)$ by $(x'_i, y'_i)$, where $(x'_i, y'_i)
\lessdot_P (x_i, y_i)$.  This condition means that we either have (a) $x_i'=x_i$
and $y_i'=y_i-1$, or (b) $x_i'=x_i-1$, $y_i'=y_i$. In case (a), we have
$\phi_1(A')=\phi_1(A)$ and $x_1<x_2<\cdots<x_{i-1}<x_i'<x_{i+1}<\cdots<x_k$,
which implies that $y_1>y_2>\cdots>y_{i-1}>y_i'>y_{i+1}>\cdots>y_k$ by Lemma
\ref{lem:jab_antichain}.  Since $y_i'=y_i-1$, it follows that
$\phi_2(A')<\phi_2(A)$ is a covering relation in $\calc(b, k)$.  Similarly, in
case (b), we have $\phi_2(A')=\phi_2(A)$ and $\phi_1(A')<\phi_1(A)$ is a
covering relation in $\calc(a, k)$.  In either case, $\phi(A') < \phi(A)$ is a
covering relation $\calc(a, k) \times \calc(b, k)$.

Conversely, suppose that $(X', Y') \lessdot_{\calc(a, k) \times \calc(b, k)} (X,
Y)$. This implies that we either have (a) $Y' = Y$ and $X' \lessdot_{\calc(a, k)}
X$, or (b) $X' = X$ and $Y' \lessdot_{\calc(b, k)} Y$. Suppose that we are in
case (a). Lemma \ref{lem:cak} (ii) implies that if we write $X = \{x_1, x_2,
\cdots, x_k\}$ with $x_1 < x_2 < \cdots < x_k$, then there exists $i \in [k]$
such that $X' = X \backslash \{x_i\} \cup \{x'_i\}$, $x'_i = x_i - 1$, and $$
x_1 < x_2 < \cdots < x_{i-1} < x'_i < x_{i+1} < \cdots < x_k .$$ It follows that
$A = \{(x_1, y_1), \ (x_2, y_2), \cdots, (x_k, y_k)\}$ and $A' = (A \backslash
\{(x_i, y_i)\}) \cup \{(x'_i, y_i)\}$ are antichains in $\cala_k(P)$, and that
$(x'_i, y_i) \lessdot_P (x_i, y_i)$.  Proposition \ref{prop:basics} (ii) then
implies that $A' \lessdot_{\cala_k(P)} A$.  This completes the proof of case
(a). Case (b) is proved using an analogous argument, which completes the proof
of the proposition.
\end{proof}

\begin{corollary}
\label{cor:Durfee}
If $a,b,k\in \N$ satisfy $k\le\min(a,b)$, then we have $ \cala_k([a]\times
[b])\cong \dkab $ as posets.
\end{corollary}

\begin{proof}
This follows immediately from Propositions \ref{prop:Durfee} and
\ref{prop:akab}.
\end{proof}

\section{Other minuscule posets}
\label{sec:other_types}

We now investigate the posets $\akp$ for minuscule posets of other types.  We
first deal with the infinite minuscule family $P_n=J([n]\times [2])$, which
correspond to the spin representations of type $D_{n+2}$ for $n\ge 2$ and
satisfy $\width(P_n)=\lfloor (n+2)/2\rfloor$. 

\begin{proposition}
    \label{prop:Dspin}
    If $n,k\in \N$ satisfy $0\le k\le \lfloor (n+2)/2\rfloor$, then we have
    isomorphisms of posets \[ \cala_k(J([n]\times[2]))\cong
    \cala_k(\calc(n+2,2))\cong \calc(n+2,2k).  \]
\end{proposition}

\begin{proof}
Since $J([n]\times [2])\cong \calc(n+2,2)$ by Proposition \ref{prop:jab_iso}, it
suffices to prove $\cala_k(\calc(n+2,2))\cong \calc(n+2,2k)$. Let
$P=\calc(n+2,2)$ and write every element of $P$ in the form $(x,y)$ where $x<y$
as in Definition \ref{def:cak}. Then we may view $P$ as embedded in the lattice
$\Z_+^2$ as we did the poset $[a]\times[b]$ in the previous two sections, because two elements $p_1=(x_1,y_1)$ and
$p_2=(x_2,y_2)$ in $P$ satisfy $p_1\le_\calc p_2$ if and only if $x_1\le x_2$
and $y_1\le y_2$. It follows that the characterization of antichains in
$[a]\times [b]$ given in Lemma \ref{lem:jab_antichain}
holds for the poset $P=\calc(n+2,2)$ as well, so that
every antichain of $P$ of size $k$ can be written uniquely as 
\[
    A = \{ (x_1, y_1), \ (x_2, y_2), \ \cdots, \ (x_k, y_k)\}
\] 
where $x_1 < x_2 < \cdots < x_k$ and $y_1 > y_2 > \cdots > y_k$.
By assumption we have $x_k<y_k$, so we have 
\[
x_1<x_2<\cdots<x_k<y_k<\cdots<y_2<y_1. 
\] 
It follows that there is a function $\phi : \cala_k(P) \ra \calc(n+2, 2k)$ given
by $$ \phi(A) = (x_1, x_2, \cdots, x_k, y_k, \cdots, y_2, y_1).$$  Furthermore,
$\phi$ is a bijection, because the assignment 
\[
    (x_1, x_2,\cdots, x_k, y_k, y_{k-1},\cdots, y_1) \mapsto \{(x_1,y_1), \cdots, (x_k,y_k)\}
\]
is clearly a two-sided inverse of $\phi$ that takes $\calc(n+2,2k)$ to $\akp$.

We claim that $\phi$ is an isomorphism of posets. To see this, it is enough to
prove that $\phi$ and $\phi\inverse$ preserve covering relations. Let $A \in
\cala_k(P)$, and write $A = \{ (x_1, y_1), \ (x_2, y_2), \ \cdots, \ (x_k,
y_k)\}$, where $x_1 < x_2 < \cdots < x_k$ and $y_1 > y_2 > \cdots > y_k$.

Suppose that we have $A' \lessdot_{\cala_k(P)} A$ for some $A'\in \akp$.
Proposition \ref{prop:basics} (ii) implies that $A'$ can be obtained from $A$ by
replacing one of the $(x_i, y_i)$ by $(x'_i, y'_i)$, where $(x'_i, y'_i)
\lessdot_P (x_i, y_i)$. By Lemma \ref{lem:cak} (ii), it follows that we
either have (a) $x_i'=x_i$ and $y_i'=y_i-1$, or (b) $x_i'=x_i-1$, $y_i'=y_i$.
In the first case, we have $x_1<\cdots< x_i'=x_i< \cdots<x_k$, so that we must
have 
\[ x_1<\cdots<x_i'<\cdots< x_k<y_k<\cdots <y_i'<\cdots<y_1 \] 
by the argument from the first paragraph of this proof. It follows that 
\begin{eqnarray*} 
\varphi(A')&=&(x_1,\cdots,x_i, \cdots,x_k, y_k, \cdots,y_i', \cdots,  y_1) \\ 
      &\lessdot&(x_1,\cdots,x_i, \cdots,x_k, y_k, \cdots,y_i, \cdots, y_1)\\ 
      &=&\varphi(A) \end{eqnarray*} 
in $\calc(n+2,2k)$. Moreover, since $y_i'=y_i-1$, the relation
$\varphi(A')<\varphi(A)$ is a covering relation in $\calc(n+2,2k)$ by Lemma
\ref{lem:cak} (ii).  Similarly, a symmetric argument shows that in the second
case we have 
\begin{eqnarray*} 
\varphi(A')&=&(x_1,\cdots,x_i', \cdots,x_k,y_k, \cdots,y_i, \cdots,  y_1) \\ 
&\lessdot&(x_1,\cdots,x_i, \cdots,x_k, y_k, \cdots,y_i, \cdots,  y_1)\\ 
&=&\varphi(A) 
\end{eqnarray*} 
in $\calc(n+2,2k)$. 

Conversely, suppose that $\varphi(A')\lessdot_\calc \varphi(A)$ for some $A'\in
\akp$.  Lemma \ref{lem:cak} (ii) implies that one of the following conditions
must hold for some $i\in [k]$:

\begin{enumerate}
\item we have $\varphi(A')=(x_1, \cdots, x_i'=x_i-1, \cdots, x_k,y_k,\cdots, y_1)$; 
\item we have $\varphi(A)=(x_1, \cdots, x_k,y_k,\cdots,y_i'=y_i-1,\cdots, y_1)$. 
\end{enumerate} 

In the first case, we can use the inverse of $\varphi$ mentioned earlier to
recover $A'$ as the set $$ A'=\{(x_j,y_j):j\in [k], j\neq i\}\cup\{(x'_i,y_i)\}
,$$ where $(x'_i,y_i)\lessdot_P (x_i,y_i)$ by Lemma \ref{lem:cak} (ii); therefore $A'\lessdot_{\akp} A$ by Proposition \ref{prop:basics} (ii).  A
similar argument shows that $A'\lessdot_{\akp} A$ in the second case, and we are
done.
 \end{proof}

\begin{theorem}\label{thm:dist}
Let $P$ be a minuscule poset, let $k\in \N$, and suppose $k\le
\width(P)$.
\begin{itemize}
    \item[{\rm (i)}]{If $P \cong [a] \times [b]$, then we have $\cala_k(P) \cong
         \calc(a, k) \times \calc(b, k)$ as posets.}
     \item[{\rm (ii)}]{If $P \cong J([n] \times [2])$, then we have
      $\cala_k(P)\cong \cala_k(\calc(n+2,2))\cong \calc(n+2,2k)$ as posets.}
  \item[{\rm (iii)}]{If $P \cong J^m([2]\times [2])$ where $m\in \N$, then
      $\cala_k(P)$ is a singleton if $k \in \{0, 2\}$, and $\cala_k(P)$ is
  isomorphic to $P$ if $k = 1$.}
 \item[{\rm (iv)}]{If $P \cong J^2([2] \times [3])$, then $\cala_k(P)$ is a
      singleton if $k = 0$, $\cala_k(P)$ is isomorphic to $P$ if $k = 1$, and
  $\cala_k(P)$ is isomorphic to $J^3([2] \times [2])$ if $k = 2$.}
\item[{\rm (v)}]{If $P \cong J^3([2] \times [3])$, then $\cala_k(P)$ is a
     singleton if $k = 0$ or $k = 3$, and $\cala_k(P)$ is isomorphic to $P$ if
 $k = 1$ or $k = 2$.}
 \item[{\rm (vi)}]{The poset $\akp$ is a distributive lattice under the order
      $\le_k$.}
\end{itemize}
\end{theorem}

\begin{proof}
Part (i) is a restatement of Proposition \ref{prop:akab}, and (ii) is a
restatement of Proposition \ref{prop:Dspin}. Induction shows that the Hasse
diagram of the poset $J^m([2]\times [2])$ for $m\in \N$ is as depicted in Figure
\ref{fig:minuscule} (c), and the assertions in (iii), (iv) and (v) then all
follow by direct computation. In particular, we have $\cala_2(P)\cong P$ for
$P=J^3([2]\times[3])$, as illustrated in Figure \ref{fig:e7iso}.

Recall that the singleton poset, posets of the form $J(P)$ where $P$ is a finite
poset, and products of distributive lattices are distributive lattices, as are
the binomial posets $\calc(n,k)$ by Remark \ref{rmk:symmetry}.  By (i)--(v), every
poset of the form $\cala_k(P)$ where $P$ is a minuscule poset has one of the
forms described above, and (vi) follows.
\end{proof}

\begin{figure}[h!]
\centering
\subfloat
{
\begin{tikzpicture} 
\node[main node] (0) {}; 
\node[main node] (1) [above right=0.5cm and 0.5cm of 0] {};
\node[main node] (2) [above right=0.5cm and 0.5cm of 1] {};
\node[main node] (3) [above right=0.5cm and 0.5cm of 2] {};
\node[main node] (4) [above right=0.5cm and 0.5cm of 3] {};
\node[main node] (5) [above right=0.5cm and 0.5cm of 4] {};

\node[main node] (00) [above left=0.5cm and 0.5cm of 3] {};
\node[main node] (10) [above right=0.5cm and 0.5cm of 00] {};
\node[main node] (20) [above right=0.5cm and 0.5cm of 10] {};

\node[main node] (01) [above left=0.5cm and 0.5cm of 10] {};
\node[main node] (02) [above right=0.5cm and 0.5cm of 01] {};
\node[main node] (03) [above right=0.5cm and 0.5cm of 02] {};

\node[main node] (11) [above left=0.5cm and 0.5cm of 01] {};
\node[main node] (12) [above right=0.5cm and 0.5cm of 11] {};
\node[main node] (13) [above right=0.5cm and 0.5cm of 12] {};
\node[main node] (14) [above right=0.5cm and 0.5cm of 13] {};
\node[main node] (15) [above right=0.5cm and 0.5cm of 14] {};

\node[main node] (21) [above left=0.5cm and 0.5cm of 11] {};
\node[main node] (22) [above right=0.5cm and 0.5cm of 21] {};
\node[main node] (23) [above right=0.5cm and 0.5cm of 22] {};
\node[main node] (24) [above right=0.5cm and 0.5cm of 23] {};
\node[main node] (25) [above right=0.5cm and 0.5cm of 24] {};

\node[main node] (31) [above left=0.5cm and 0.5cm of 24] {};
\node[main node] (32) [above right=0.5cm and 0.5cm of 31] {};
\node[main node] (33) [above right=0.5cm and 0.5cm of 32] {};
\node[main node] (34) [above right=0.5cm and 0.5cm of 33] {};
\node[main node] (35) [above right=0.5cm and 0.5cm of 34] {};

\node [right=0.05cm of 0] {$a$};
\node [right=0.05cm of 1] {$b$};
\node [right=0.05cm of 2] {$c$};
\node [right=0.05cm of 3] {$d$};
\node [left=0.05cm of 00] {$e$};
\node [right=0.05cm of 4] {$f$};
\node [left=0.05cm of 10] {$g$};
\node [right=0.05cm of 5] {$h$};
\node [left=0.05cm of 01] {$i$};
\node [right=0.05cm of 20] {$j$};
\node [left=0.05cm of 11] {$k$};
\node [right=0.05cm of 02] {$l$};
\node [left=0.05cm of 21] {$m$};
\node [left=0.05cm of 12] {$n$};
\node [right=0.05cm of 03] {$o$};
\node [left=0.05cm of 22] {$p$};
\node [right=0.05cm of 13] {$q$};
\node [left=0.05cm of 23] {$r$};
\node [right=0.05cm of 14] {$s$};
\node [left=0.05cm of 24] {$t$};
\node [right=0.05cm of 15] {$u$};
\node [left=0.05cm of 31] {$v$};
\node [right=0.05cm of 25] {$w$};
\node [left=0.05cm of 32] {$x$};
\node [left=0.05cm of 33] {$y$};
\node [left=0.05cm of 34] {$z$};
\node [left=0.05cm of 35] {$\theta$};

\path[draw]
(0)--(1)--(2)--(3)--(4)--(5)
(00)--(10)--(20)
(01)--(02)--(03)
(11)--(12)--(13)--(14)--(15)
(21)--(22)--(23)--(24)--(25)
(31)--(32)--(33)--(34)--(35)
(3)--(00)
(4)--(10)--(01)--(11)--(21)
(5)--(20)--(02)--(12)--(22)
(03)--(13)--(23)
(14)--(24)--(31)
(15)--(25)--(32);
\end{tikzpicture}
}
\qquad\qquad
\subfloat{
\begin{tikzpicture} 
\node (0) {}; 
\node (c) [above = 9*0.5cm of 0] {\Large{$\cong$}};
\end{tikzpicture}
}
\qquad\qquad
\subfloat{
\begin{tikzpicture} 
\node[main node] (0) {}; 
\node[main node] (1) [above right=0.5cm and 0.5cm of 0] {};
\node[main node] (2) [above right=0.5cm and 0.5cm of 1] {};
\node[main node] (3) [above right=0.5cm and 0.5cm of 2] {};
\node[main node] (4) [above right=0.5cm and 0.5cm of 3] {};
\node[main node] (5) [above right=0.5cm and 0.5cm of 4] {};

\node[main node] (00) [above left=0.5cm and 0.5cm of 3] {};
\node[main node] (10) [above right=0.5cm and 0.5cm of 00] {};
\node[main node] (20) [above right=0.5cm and 0.5cm of 10] {};

\node[main node] (01) [above left=0.5cm and 0.5cm of 10] {};
\node[main node] (02) [above right=0.5cm and 0.5cm of 01] {};
\node[main node] (03) [above right=0.5cm and 0.5cm of 02] {};

\node[main node] (11) [above left=0.5cm and 0.5cm of 01] {};
\node[main node] (12) [above right=0.5cm and 0.5cm of 11] {};
\node[main node] (13) [above right=0.5cm and 0.5cm of 12] {};
\node[main node] (14) [above right=0.5cm and 0.5cm of 13] {};
\node[main node] (15) [above right=0.5cm and 0.5cm of 14] {};

\node[main node] (21) [above left=0.5cm and 0.5cm of 11] {};
\node[main node] (22) [above right=0.5cm and 0.5cm of 21] {};
\node[main node] (23) [above right=0.5cm and 0.5cm of 22] {};
\node[main node] (24) [above right=0.5cm and 0.5cm of 23] {};
\node[main node] (25) [above right=0.5cm and 0.5cm of 24] {};

\node[main node] (31) [above left=0.5cm and 0.5cm of 24] {};
\node[main node] (32) [above right=0.5cm and 0.5cm of 31] {};
\node[main node] (33) [above right=0.5cm and 0.5cm of 32] {};
\node[main node] (34) [above right=0.5cm and 0.5cm of 33] {};
\node[main node] (35) [above right=0.5cm and 0.5cm of 34] {};

\node [right=0.05cm of 0] {$ef$};
\node [right=0.05cm of 1] {$eh$};
\node [right=0.05cm of 2] {$gh$};
\node [right=0.05cm of 3] {$ih$};
\node [left=0.05cm of 00] {$ij$};
\node [right=0.05cm of 4] {$kh$};
\node [left=0.05cm of 10] {$kj$};
\node [right=0.05cm of 5] {$mh$};
\node [left=0.05cm of 01] {$kl$};
\node [right=0.05cm of 20] {$mj$};
\node [left=0.05cm of 11] {$ko$};
\node [right=0.05cm of 02] {$ml$};
\node [left=0.05cm of 21] {$no$};
\node [left=0.05cm of 12] {$mo$};
\node [right=0.05cm of 03] {$mn$};
\node [left=0.05cm of 22] {$po$};
\node [right=0.05cm of 13] {$mq$};
\node [left=0.05cm of 23] {$pq$};
\node [right=0.05cm of 14] {$ms$};
\node [left=0.05cm of 24] {$ps$};
\node [right=0.05cm of 15] {$mu$};
\node [left=0.05cm of 31] {$rs$};
\node [right=0.05cm of 25] {$pu$};
\node [left=0.05cm of 32] {$ru$};
\node [left=0.05cm of 33] {$tu$};
\node [left=0.05cm of 34] {$vu$};
\node [left=0.05cm of 35] {$vw$};
\path[draw]
(0)--(1)--(2)--(3)--(4)--(5)
(00)--(10)--(20)
(01)--(02)--(03)
(11)--(12)--(13)--(14)--(15)
(21)--(22)--(23)--(24)--(25)
(31)--(32)--(33)--(34)--(35)
(3)--(00)
(4)--(10)--(01)--(11)--(21)
(5)--(20)--(02)--(12)--(22)
(03)--(13)--(23)
(14)--(24)--(31)
(15)--(25)--(32);
\end{tikzpicture}
}
\caption{Isomorphism between $P=J^3([2]\times[3])$ and $\cala_2(P)$, with each
element $\{\alpha,\beta\}\in \cala_2(P)$ written as $\alpha\beta$.}
\label{fig:e7iso}
\end{figure}

\section{Branching rules of minuscule representations}
\label{sec:repthy}

Recall from the introduction that the minuscule posets are precisely the posets
$P$ such that the poset $J(P)$ appears as the weight poset of a minuscule
representation of a simple Lie algebra. The goal of this section is to explain a
remarkable connection in this setting between the sets $\akp$ and branching
rules of suitable restrictions of the associated minuscule representations. We
fix $\C$ as the ground field throughout the section.

Let $\fg$ be a simple Lie algebra over $\C$ of Dynkin type $X$ and rank $n$.  It
is known that a minuscule representation exists for $\fg$ if and only if $\fg$
has type $A_{n} (n\ge 1), D_n (n\ge 4), E_6,$ or $E_7$, and that in these types
each minuscule representation has a fundamental weight $\omega_p$ for some $1\le
p\le n$ as its highest weight; we will denote such a minuscule representation by
$\lxp$.  The Dynkin diagrams of the aforementioned types are shown in Figure
\ref{fig:dynkin} below, where we adopt the labelling conventions of
\cite[Chapter 4]{kac90}. The precise values of $p$ corresponding to the possible
minuscule representations are listed in Table \ref{tab:minu}, along with the
isomorphism type of the minuscule poset $P$ for which the weight poset of $\lxp$
is isomorphic to $J(P)$; see also Figure \ref{fig:minuscule} and \cite[Theorem
8.3.10 (v)]{green13}.  We denote the minuscule poset $P$ corresponding to $\lxp$
by $\pxp$ from now on.
\begin{figure}[h!] \centering \subfloat[{$A_n$}] {
        \begin{tikzpicture}

    \node[minu node] (1) {};
            \node[minu node] (2) [right=1cm of 1] {};
            \node[minu node] (3) [right=1.5cm of 2] {};
            \node[minu node] (4) [right=1cm of 3] {};

            \node (11) [below=0.1cm of 1] {\small{$1$}};
            \node (22) [below=0.1cm of 2] {\small{$2$}};
            \node (33) [below=0.1cm of 3] {\small{$n-1$}};
            \node (44) [below=0.15cm of 4] {\small{$n$}};

            \path[draw]
            (1)--(2)
            (3)--(4);

            \path[draw,dashed]
            (2)--(3);
\end{tikzpicture}
}
\quad\quad\quad
\subfloat[{$D_n$}]
{
\begin{tikzpicture}

            \node[minu node] (1) {};
            \node[main node] (2) [right=1cm of 1] {};
            \node[main node] (3) [right=1.5cm of 2] {};
            \node[main node] (4) [right=1cm of 3] {};
            \node[minu node] (5) [above right=0.7cm and 0.9cm of 4] {};
            \node[minu node] (6) [below right=0.7cm and 0.9cm of 4] {};

            \node (11) [below=0.1cm of 1] {\small{$1$}};
            \node (22) [below=0.1cm of 2] {\small{$2$}};
            \node (33) [below=0.1cm of 3] {\small{$n-3$}};
            \node (44) [right=0.1cm of 4] {\small{$n-2$}};
            \node (55) [right=0.1cm of 5] {\small{$n-1$}};
            \node (66) [right=0.1cm of 6] {\small{$n$}};

            \path[draw]
            (1)--(2)
            (3)--(4)--(5)
            (4)--(6);

            \path[draw,dashed]
            (2)--(3);
\end{tikzpicture}
}\\
\subfloat[{$E_6$}]
{
\begin{tikzpicture}
    \node[minu node] (1) {};
    \node[main node] (2) [right=1cm of 1] {};
            \node[main node] (3) [right=1cm of 2] {};
            \node[main node] (4) [right=1cm of 3] {};
            \node[minu node] (5) [right=1cm of 4] {};
            \node[main node] (6) [above=1cm of 3] {};

            \path[draw]
            (1)--(2)--(3)--(4)--(5)
            (3)--(6);

            \node (11) [below=0.1cm of 1] {\small{$1$}};
            \node (22) [below=0.1cm of 2] {\small{$2$}};
            \node (33) [below=0.1cm of 3] {\small{$3$}};
            \node (44) [below=0.1cm of 4] {\small{$4$}};
            \node (55) [below=0.1cm of 5] {\small{$5$}};
            \node (66) [above=0.1cm of 6] {\small{$6$}};
\end{tikzpicture}
}
\quad\quad\quad
\subfloat[{$E_7$}]
{
\begin{tikzpicture}
    \node[main node] (1) {};
    \node[main node] (2) [right=1cm of 1] {};
            \node[main node] (3) [right=1cm of 2] {};
            \node[main node] (4) [right=1cm of 3] {};
            \node[main node] (5) [right=1cm of 4] {};
            \node[minu node] (6) [right=1cm of 5] {};
            \node[main node] (7) [above=1cm of 3] {};

            \path[draw]
            (1)--(2)--(3)--(4)--(5)--(6)
            (3)--(7);

            \node (11) [below=0.1cm of 1] {\small{$1$}};
            \node (22) [below=0.1cm of 2] {\small{$2$}};
            \node (33) [below=0.1cm of 3] {\small{$3$}};
            \node (44) [below=0.1cm of 4] {\small{$4$}};
            \node (55) [below=0.1cm of 5] {\small{$5$}};
            \node (66) [below=0.1cm of 6] {\small{$6$}};
            \node (77) [above=0.1cm of 7] {\small{$7$}};
\end{tikzpicture}
}
\caption{Simple Lie algebras admitting minuscule
representations.}
\label{fig:dynkin}
\end{figure}

\begin{table}[h!] \centering
    \begin{tabularx}{20em}{c@{\hspace{2.8em}}c@{\hspace{3.6em}}c}
         \toprule
         X  &  $p$ &  $\pxp$\\
         \midrule
         $A_n$ & $k (k\in [n])$ & $[k]\times[n+1-k]$\\
         $D_n$ & $1$ & $J^{n-3}([2]\times[2])$ \\
         $D_n$ & $n-1$ or $n$ & $J([n-2]\times[2])$ \\
         $E_6$ & 1 or 5 & $J^2([2]\times[3])$ \\
         $E_7$ & 6 & $J^3([2]\times[3])$\\
         \bottomrule
\end{tabularx}  
\caption{Classification of minuscule posets.} 
          \label{tab:minu}
      \end{table}

In addition to recovering the poset structure of the weight poset of $\lxp$,
the minuscule poset $\pxp$ can in fact also be used to simultaneously construct
$\lxp$ itself and the Lie algebra $\fg$; see \cite{wildberger,whatis}. The
construction involves turning the poset $\pxp$ into a so-called heap (in the
sense of Viennot \cite{viennot}) and then defining linear operators $X_i, Y_i,
H_i\,(1\le i\le n)$ on the free vector space $V$ spanned by the ideals of the
heap; we denote this heap by $H(X,\omega_p)$. The linear operators $X_i, Y_i,
H_i (1\le i\le n)$ generate an isomorphic copy of $\fg$ inside the Lie algebra
$\mathfrak{gl}(V)$, and the construction endows $V$ with the structure of a
$\fg$-module affording the minuscule representation $\lxp$.  In addition, the
construction has the convenient feature that $p$ will appear as the label of the
unique maximal element of $\pxp$ in the heap $\hxp$, allowing us to quickly
identify the highest weight of the minuscule representation associated with the
heap as $\omega_p$. For more details about the construction, we refer the reader
to \cite{wildberger,whatis,green13}.

\begin{example}
    \label{eg:heaps}
Figure \ref{fig:heaps} shows the heaps corresponding to the minuscule
representations $L(E_6,\omega_1)$ and $L(E_6,\omega_5)$, which are not
isomorphic as heaps because of their different labellings but have isomorphic
underlying posets $P(E_6,\omega_1)\cong P(E_6,\omega_5)\cong J^2([2]\times[3])$.

\begin{figure}[h!]
\centering
\subfloat[{$H(E_6,\omega_1)$}]
{
\begin{tikzpicture} 
\node[main node] (0) {}; 
\node[main node] (1) [above right=0.5cm and 0.5cm of 0] {};
\node[main node] (2) [above right=0.5cm and 0.5cm of 1] {};
\node[main node] (3) [above right=0.5cm and 0.5cm of 2] {};
\node[main node] (4) [above right=0.5cm and 0.5cm of 3] {};

\node[main node] (00) [above left=0.5cm and 0.5cm of 2] {};
\node[main node] (10) [above right=0.5cm and 0.5cm of 00] {};
\node[main node] (20) [above right=0.5cm and 0.5cm of 10] {};

\node[main node] (01) [above left=0.5cm and 0.5cm of 10] {};
\node[main node] (02) [above right=0.5cm and 0.5cm of 01] {};
\node[main node] (03) [above right=0.5cm and 0.5cm of 02] {};

\node[main node] (11) [above left=0.5cm and 0.5cm of 01] {};
\node[main node] (12) [above right=0.5cm and 0.5cm of 11] {};
\node[main node] (13) [above right=0.5cm and 0.5cm of 12] {};
\node[main node] (14) [above right=0.5cm and 0.5cm of 13] {};
\node[main node] (15) [above right=0.5cm and 0.5cm of 14] {};

\node [right=0.05cm of 0] {$5$};
\node [right=0.05cm of 1] {$4$};
\node [right=0.05cm of 2] {$3$};
\node [right=0.05cm of 3] {$2$};
\node [right=0.05cm of 4] {$1$};
\node [left=0.05cm of 00] {$6$};
\node [left=0.05cm of 10] {$3$};
\node [right=0.05cm of 20] {$2$};
\node [left=0.05cm of 01] {$4$};
\node [right=0.05cm of 02] {$3$};
\node [right=0.05cm of 03] {$6$};
\node [left=0.05cm of 11] {$5$};
\node [left= 0.05cm of 12] {$4$};
\node [left=0.05cm of 13] {$3$};
\node [left=0.05cm of 14] {$2$};
\node [left=0.05cm of 15] {$1$};

\path[draw]
(0)--(1)--(2)--(3)--(4)
(00)--(10)--(20)
(01)--(02)--(03)
(11)--(12)--(13)--(14)--(15)
(2)--(00)
(3)--(10)--(01)--(11)
(4)--(20)--(02)--(12)
(03)--(13);

\end{tikzpicture}

}
\qquad\qquad
\subfloat[{$H(E_6,\omega_5)$}]
{
\begin{tikzpicture} 
\node[main node] (0) {}; 
\node[main node] (1) [above right=0.5cm and 0.5cm of 0] {};
\node[main node] (2) [above right=0.5cm and 0.5cm of 1] {};
\node[main node] (3) [above right=0.5cm and 0.5cm of 2] {};
\node[main node] (4) [above right=0.5cm and 0.5cm of 3] {};

\node[main node] (00) [above left=0.5cm and 0.5cm of 2] {};
\node[main node] (10) [above right=0.5cm and 0.5cm of 00] {};
\node[main node] (20) [above right=0.5cm and 0.5cm of 10] {};

\node[main node] (01) [above left=0.5cm and 0.5cm of 10] {};
\node[main node] (02) [above right=0.5cm and 0.5cm of 01] {};
\node[main node] (03) [above right=0.5cm and 0.5cm of 02] {};

\node[main node] (11) [above left=0.5cm and 0.5cm of 01] {};
\node[main node] (12) [above right=0.5cm and 0.5cm of 11] {};
\node[main node] (13) [above right=0.5cm and 0.5cm of 12] {};
\node[main node] (14) [above right=0.5cm and 0.5cm of 13] {};
\node[main node] (15) [above right=0.5cm and 0.5cm of 14] {};

\node [right=0.05cm of 0] {$1$};
\node [right=0.05cm of 1] {$2$};
\node [right=0.05cm of 2] {$3$};
\node [right=0.05cm of 3] {$4$};
\node [right=0.05cm of 4] {$5$};
\node [left=0.05cm of 00] {$6$};
\node [left=0.05cm of 10] {$3$};
\node [right=0.05cm of 20] {$4$};
\node [left=0.05cm of 01] {$2$};
\node [right=0.05cm of 02] {$3$};
\node [right=0.05cm of 03] {$6$};
\node [left=0.05cm of 11] {$1$};
\node [left= 0.05cm of 12] {$2$};
\node [left=0.05cm of 13] {$3$};
\node [left=0.05cm of 14] {$4$};
\node [left=0.05cm of 15] {$5$};

\path[draw]
(0)--(1)--(2)--(3)--(4)
(00)--(10)--(20)
(01)--(02)--(03)
(11)--(12)--(13)--(14)--(15)
(2)--(00)
(3)--(10)--(01)--(11)
(4)--(20)--(02)--(12)
(03)--(13);
\end{tikzpicture}

}
\caption{Heaps for the minuscule representations $L(E_6,\omega_1)$ and
$L(E_6,\omega_5)$.}
\label{fig:heaps}
\end{figure}
\end{example}

\begin{remark}
    \label{rmk:automorphism}
Example \ref{eg:heaps} illustrates the fact that in general a minuscule
representation $\lxp$ (or, equivalently, the corresponding heap $\hxp$) contains
strictly more data than those of the associated minuscule poset $\pxp$ due to
choices in labelling. On the other hand, we note that the only cases where
distinct values $p,p'$ give rise to isomorphic minuscule posets
$P(X,\omega_p)\cong P(X,\omega_{p'})$ are as follows: when $X=A_n$ and $p+p'=n$,
or when $X=D_4$ and $2\notin \{p,p'\}$, or when $X=D_n$ for some $n>4$ and
$\{p,p'\}=\{n-1,n\}$ as sets, or when $X=E_6$ and $\{p,p'\}=\{1,5\}$ as sets. In
all these cases, $p$ and $p'$ are conjugate by an involutive automorphism of the
Dynkin diagram, and consequently the minuscule representations $L(X,\omega_p)$
and $L(X,\omega_p')$ differ only by the corresponding automorphism of the Lie
algebra $\fg$. It follows that although a minuscule poset  generally contains
less information than the representation it arises from, one can still recover
minuscule representations from minuscule posets up to diagram automorphisms of
$\fg$ in all cases. \end{remark}

Let $L=\lxp$ be a minuscule representation of a simple Lie algebra $\fg$.  Let
$\fk$ be the subalgebra of $\fg$ generated by the set $\{e_i, f_i, h_i : 1 \leq
i \leq n, i \ne p\}$, let $Y$ be the Dynkin type of $\fk$, and consider the
restriction $L\!\downarrow_\fk$ of the module $L$ to $\fk$. By
\cite[Proposition 8.2.9 (iv)]{green13}, the decomposition of $\res$ into simple
components can be described in a uniform way via the heap $H=\hxp$, as follows.
Denote the elements in $H$ with label $p$ by $\pi_1,\cdots,\pi_k$, ordered so
that $\pi_1\le \pi_2\le \cdots \le \pi_k$ in the minuscule poset $P=\pxp$. For
each $0\le i\le k$, let $H'_i$ be the subheap of $H$ consisting of all
elements $\alpha$ such that $\alpha\not \le \pi_{i}$ and $\pi_{i+1}\not\le
\alpha$, where the former condition holds vacuously for $i=0$ and the latter
condition holds vacuously for $i=k$.  In other words, we take $H'_i$ to be the
subheap of $H$ obtained by removing all elements in the order ideal generated by
$\pi_{i}$ (which is considered empty if $i=0$) or in the order filter generated
by $\pi_{i+1}$ (which is considered empty
if $i=k$). The subheap $H'_i$ will not contain any element labelled by $p$,
and it turns out that $H'_i$ will coincide with the heap of a minuscule
representation, $L'_i$, of $\fk$. Moreover, we have $L\!
\downarrow_\fk=\oplus_{i=0}^{k} L'_i$ as $\fk$-modules, where $L'_i$ is the trivial representation of $\fk$ if
$H'_i$ is empty (\cite[Proposition
8.2.9 (iv)]{green13}). Note that as remarked in the paragraph above Example
\ref{eg:heaps}, it is easy to identify which
minuscule representation of $\fk$ each $L'_i$ is by reading the label of the top
element in $H'_i$ (whenever $H'_i$ is nonempty).

\begin{example}
\label{eg:e7branching}
Figure \ref{fig:e7branching} shows the heap $H=H(E_7,\omega_7)$ associated to
the minuscule representation $L$ of fundamental weight $\omega_7$ for the Lie
algebra $\fg$ of type $E_7$. We can determine the branching rule for the
restriction of $L$ to the type-$E_6$ subalgebra $\fk$ generated by
$\{e_i,f_i,h_i: i\in [7], i\neq 6\}$ as follows. The heap $H$ contains $k=3$
elements with label $7$, which are located at the top, bottom, and the middle
left of the heap, so the restriction $L\!\downarrow_\fk$ contains $k+1=4$
simple components, $L'_0,L'_1,L'_2$ and $L_3'$. The sets $H'_1$ and $H'_3$ are
empty, so the components $L'_0$ and $L'_3$ both afford the trivial
representation.  The subheap $H'_1$ is colored in red, and may be characterized
as the interval $[a,b]=\{x\in H: a\le x\le b\}$ where $a$ is the unique minimal
element with label 5 in $H$ and $b$ is the unique maximal element with label 1 in $H$.
Since the top element of $H'_1$ has label $1$, the module $L'_1$ has highest
weight $\omega_1$ and thus affords the minuscule representation
$L(E_6,\omega_1)$. (Note that in the labelling convention for type
$E_6$, we should identify the vertex $7$ from the $E_7$ Dynkin diagram as the
vertex $6$ in the $E_6$ diagram; see Figure \ref{fig:dynkin}.) Similarly, the
subheap $H'_2$ contains the interval in $H$ from the unique minimal element
labelled by 1 to the unique maximal element labelled by 5, so that $L'_3$ is a
copy of the minuscule representation $L(E_6,\omega_5)$ of $\fk$.  The minuscule
posets underlying the heaps $H'_1$ and $H'_2$ are both isomorphic to
$J^2([2]\times[3])$, so it follows from Theorem \ref{thm:dist} (iv) that for all
$0\le i\le k$, the weight poset of $L'_i$ is isomorphic to $\mathcal{A}_i(P)$
for the minuscule poset $P=J^2([2]\times [3])$.

\begin{figure}[h!]
\centering
\begin{tikzpicture} 
\node[minu node] (0) {}; 
\node[red node] (1) [above right=0.5cm and 0.5cm of 0] {};
\node[red node] (2) [above right=0.5cm and 0.5cm of 1] {};
\node[red node] (3) [above right=0.5cm and 0.5cm of 2] {};
\node[red node] (4) [above right=0.5cm and 0.5cm of 3] {};
\node[red node] (5) [above right=0.5cm and 0.5cm of 4] {};

\node[red node] (00) [above left=0.5cm and 0.5cm of 3] {};
\node[red node] (10) [above right=0.5cm and 0.5cm of 00] {};
\node[red node] (20) [above right=0.5cm and 0.5cm of 10] {};

\node[red node] (01) [above left=0.5cm and 0.5cm of 10] {};
\node[red node] (02) [above right=0.5cm and 0.5cm of 01] {};
\node[red node] (03) [above right=0.5cm and 0.5cm of 02] {};

\node[red node] (11) [above left=0.5cm and 0.5cm of 01] {};
\node[red node] (12) [above right=0.5cm and 0.5cm of 11] {};
\node[red node] (13) [above right=0.5cm and 0.5cm of 12] {};
\node[red node] (14) [above right=0.5cm and 0.5cm of 13] {};
\node[red node] (15) [above right=0.5cm and 0.5cm of 14] {};

\node[minu node] (21) [above left=0.5cm and 0.5cm of 11] {};
\node[main node] (22) [above right=0.5cm and 0.5cm of 21] {};
\node[main node] (23) [above right=0.5cm and 0.5cm of 22] {};
\node[main node] (24) [above right=0.5cm and 0.5cm of 23] {};
\node[main node] (25) [above right=0.5cm and 0.5cm of 24] {};

\node[main node] (31) [above left=0.5cm and 0.5cm of 24] {};
\node[main node] (32) [above right=0.5cm and 0.5cm of 31] {};
\node[main node] (33) [above right=0.5cm and 0.5cm of 32] {};
\node[main node] (34) [above right=0.5cm and 0.5cm of 33] {};
\node[minu node] (35) [above right=0.5cm and 0.5cm of 34] {};

\node [right=0.05cm of 0] {$6$};
\node [red,right=0.05cm of 1] {$5$};
\node [red,right=0.05cm of 2] {$4$};
\node [red,right=0.05cm of 3] {$3$};
\node [red,left=0.05cm of 00] {$7$};
\node [red,right=0.05cm of 4] {$2$};
\node [red,left=0.05cm of 10] {$3$};
\node [red,right=0.05cm of 5] {$1$};
\node [red,left=0.05cm of 01] {$4$};
\node [red,right=0.05cm of 20] {$2$};
\node [red,left=0.05cm of 11] {$5$};
\node [red,right=0.05cm of 02] {$3$};
\node [left=0.05cm of 21] {$6$};
\node [red,left=0.05cm of 12] {$4$};
\node [red,right=0.05cm of 03] {$7$};
\node [left=0.05cm of 22] {$5$};
\node [red,right=0.05cm of 13] {$3$};
\node [left=0.05cm of 23] {$4$};
\node [red,right=0.05cm of 14] {$2$};
\node [left=0.05cm of 24] {$3$};
\node [red,right=0.05cm of 15] {$1$};
\node [left=0.05cm of 31] {$7$};
\node [right=0.05cm of 25] {$2$};
\node [left=0.05cm of 32] {$3$};
\node [left=0.05cm of 33] {$4$};
\node [left=0.05cm of 34] {$5$};
\node [left=0.05cm of 35] {$6$};

\path[draw,red]
(1)--(2)--(3)--(4)--(5)
(00)--(10)--(20)
(01)--(02)--(03)--(13)
(11)--(12)--(13)--(14)--(15)
(3)--(00)
(4)--(10)--(01)--(11)
(5)--(20)--(02)--(12);

\path[draw]
(0)--(1)
(11)--(21)
(12)--(22)
(21)--(22)--(23)--(24)--(25)
(31)--(32)--(33)--(34)--(35)
(13)--(23)
(14)--(24)--(31)
(15)--(25)--(32);
\end{tikzpicture}
\caption{The heap of the minuscule representation $L(E_7,6)$.}
\label{fig:e7branching}
\end{figure}
\end{example}

The main theorem of this section asserts that the poset isomorphisms observed at
the end of Example \ref{eg:e7branching} in fact hold for all minuscule
representations. The assertion is remarkable in the sense that it shows that on
the level of the minuscule posets the branching rules of minuscule
representations can be described purely in terms of the partial order $\le_k$,
even though the branching rules described via heaps depend heavily on the
labelling of the minuscule posets by the vertices of the Dynkin diagram. In
light of the last sentence of Remark \ref{rmk:automorphism}, this means that the
order $\le_k$ provides an efficient algorithm for computing the branching rule
of minuscule representations up to diagram automorphisms of Lie algebras that
does not require the use of heaps.

\begin{theorem}\label{thm:repthy}
Let $L=\lxp$ be a minuscule representation of a simple Lie algebra $\fg$ and let
$P=\pxp$ be the corresponding minuscule poset, so that the poset of weights of $L$
is isomorphic to $J(P)$. Let $\fk$ be the subalgebra of $\fg$ generated by the
set $\{e_i,f_i,h_i: 1\le i\le n, i\neq p\}$.  Then the $\fk$-module $\res$
decomposes into a direct sum \[ \res \cong \bigoplus_{i=0}^k V_i,\] where
$k=\width(P)$ and $V_i$ is a minuscule representation of $\fk$ whose weight
poset is isomorphic to $\cala_i(P)$ for each integer $0\le i\le k$.
\end{theorem}

\begin{proof}
For the unique minuscule representation $L=L(E_7,\omega_7)$ in type $E_7$, the
minuscule poset $P$ is isomorphic to $J^3([2]\times [3])$ by Table
\ref{tab:minu} (or Figure \ref{fig:e7branching}), and the theorem follows from
the discussions in Example \ref{eg:e7branching}.  For the other minuscule
representations listed in Table \ref{tab:minu}, the theorem can be proved
similarly case by case, by working out the direct summands of $\res$ via heaps
using \cite[Proposition 8.2.9 (iv)]{green13} and then verifying the facts that
$k=\width(P)$ and $P_i\cong \cala_i(P)$ for all $0\le i\le k$.  The direct sum
decompositions have also been described in \cite[Section 8.2]{green13}, so we
will simply sketch the key facts for each minuscule representation below.  We
denote the Dynkin type of $\fk$ by $Y$ in each case.

If $X=E_6$ and $p=5$, then we have $Y=D_5$ and $P=J^2([2]\times [3])$, where $P$
has width $2$ by Figure \ref{fig:minuscule} (d). Inspecting the heaps
corresponding to $L$ given in Figure \ref{fig:heaps}, we see that the restriction
$\res$ decomposes into the direct sum of the trivial module, the module
$L(D_5,\omega_1)$ (affording the natural representation with dimension 10), and
the module $L(D_5,\omega_{4})$ (affording a half spin representation with
dimension 16); see also \cite[Exercise 8.2.17]{green13}. The minuscule posets
corresponding to these summands are the empty poset, $J([2]\times[3])$, and
$J^2([2]\times [2])$, so that the weight poset of the summands are the singleton
poset, $J^2([2]\times[3])$, and $J^3([2]\times[2])$, respectively. The
conclusions of the theorem now follow from Theorem \ref{thm:dist} (iv), and a
similar argument shows that they also hold if $X=E_6$ and $p=1$.

If $X=D_n$ and $p\in \{n-1,n\}$, then we have $Y=A_{n-1}$ and
$P=J([n-2]\times[2])$, where $P$ has width $\lfloor n/2\rfloor$ by Figure
\ref{fig:minuscule} (b). In this case it is known that $L$ is one of the two
half spin representations of type $D_n$, and that (see \cite[Exercise
8.2.15]{green13}) we have an isomorphism of $\fk$-modules 
\[ \res \cong
    \bigoplus_{i = 0}^{\lfloor n/2 \rfloor} L(A_{n-1},\omega_{2i}).
\]
The minuscule poset corresponding to each summand $L(A_{n-1},\omega_{2i})$ is
$[n-2i]\times [2i]$. The weight posets of the summands are thus the posets
$J([n-2i]\times[2i])$ for $0\le i\le 2k$. 
On the other hand,
by
Proposition \ref{prop:Dspin}
and Proposition \ref{prop:jab_iso} we have  $$ \cala_i(P) \cong \calc(n, 2i) \cong
J([n-2i] \times [2i]) $$
for each $i$, so the conclusions of the theorem also hold in this case.

If $X=D_n$ and $p=1$, then we have $Y=D_{n-1}$ and $P=J^{n-3}([2]\times[2])$,
which has width $2$ by Figure \ref{fig:minuscule} (c). In
this case, it is known that $L_p$ is the natural representation of dimension
$2n$ in type $D_n$, and  that (see \cite[Exercise 8.2.16 (iii)]{green13}) $\res$
decomposes as the direct sum of three simple $\fk$-modules: two copies of the
trivial representation, and one copy of the natural representation of $\fk$. The
conclusions of the theorem now follow from Theorem \ref{thm:dist} (iii).

Finally, if $X=A_n$ and $p\in [n]$, then we have $P=[p]\times[n+1-p]$, where $P$
has width $k=\min(p,n+1-p)$ by Figure \ref{fig:minuscule} (a). In this case, we
have $Y=A_{p-1}\times A_{n-p}$, so that $\fk$ is a simple Lie algebra of type
$A_{n-1}$ if $p\in \{1,n\}$ and is the direct sum of two simple Lie algebras, one
of type $A_{p-1}$ and one of type $A_{n-p}$, otherwise.  As illustrated by
Figure \ref{fig:a},  when the Hasse diagram of the heap is drawn as in Figure
\ref{fig:minuscule} (a), the heap $H=H(A_n,\omega_p)$ has its leftmost element
labelled by 1, and the labels of the elements increase every time one moves
northeast or southeast by one step in the grid; see also \cite[Section
6.2]{green13}. The heap contains $k$ elements labelled by $p$, with the maximal
such element $\pi_k$ being the top element of $H$ and  the other
$k-1$ elements $\pi_{k-1},\cdots,\pi_1$ being the elements directly below
$\pi_k$ in the Hasse diagram.  It follows that as $i$ ranges over
the list
$k,k-1,\cdots,1,0$, the subheaps $H'_i$ of $H(A_n,\omega_p)$ corresponding to
the summands of $\res$ satisfy natural poset isomorphisms $ J(H'_i)\cong
\cald_{k-i}([p]\times[n+1-p])=\cald_{k-i}(P)$ for all $i$, where
$\cald_i([p]\times[n+1-p])$ is as
defined in Definition \ref{def:Durfee}. The assertions of the theorem now follow
from Theorem \ref{thm:dist} (i), completing the proof.
\begin{figure}[h!] \centering
    \begin{tikzpicture} \node[main node] (0) {}; \node[main node] (1) [above
          left=0.4cm and 0.4cm of 0] {}; \node[main node] (2) [above left=0.4cm
          and 0.4cm of 1] {}; \node[main node] (3) [above left=0.4cm and 0.4cm
          of 2] {};

\node[below=0.05cm of 0] {\tiny{$4$}};
\node[below=0.05cm of 1] {\tiny$3$};
\node[below=0.05cm of 2] {\tiny$2$};
\node[below=0.05cm of 3] {\tiny$1$};

\node[minu node] (4) [above right=0.4cm and 0.4cm of 0] {};
\node[main node] (5) [above right=0.4cm and 0.4cm of 4] {};
\node[main node] (6) [above right=0.4cm and 0.4cm of 5] {};
\node[main node] (7) [above right=0.4cm and 0.4cm of 6] {};

\node[below=0.05cm of 4] {\tiny$5$};
\node[below=0.05cm of 5] {\tiny$6$};
\node[below=0.05cm of 6] {\tiny$7$};
\node[below=0.05cm of 7] {\tiny$8$};

\node[main node] (40) [above right=0.4cm and 0.4cm of 1] {};
\node[minu node] (50) [above right=0.4cm and 0.4cm of 40] {};
\node[main node] (60) [above right=0.4cm and 0.4cm of 50] {};
\node[main node] (70) [above right=0.4cm and 0.4cm of 60] {};

\node[below=0.05cm of 40] {\tiny$4$};
\node[below=0.05cm of 50] {\tiny$5$};
\node[below=0.05cm of 60] {\tiny$6$};
\node[below=0.05cm of 70] {\tiny$7$};

\node[main node] (400) [above right=0.4cm and 0.4cm of 2] {};
\node[main node] (500) [above right=0.4cm and 0.4cm of 400] {};
\node[minu node] (600) [above right=0.4cm and 0.4cm of 500] {};
\node[main node] (700) [above right=0.4cm and 0.4cm of 600] {};

\node[below=0.05cm of 400] {\tiny$3$};
\node[below=0.05cm of 500] {\tiny$4$};
\node[below=0.05cm of 600] {\tiny$5$};
\node[below=0.05cm of 700] {\tiny$6$};

\node[main node] (4000) [above right=0.4cm and 0.4cm of 3] {};
\node[main node] (5000) [above right=0.4cm and 0.4cm of 4000] {};
\node[main node] (6000) [above right=0.4cm and 0.4cm of 5000] {};
\node[minu node] (7000) [above right=0.4cm and 0.4cm of 6000] {};

\node[below=0.05cm of 4000] {\tiny$2$};
\node[below=0.05cm of 5000] {\tiny$3$};
\node[below=0.05cm of 6000] {\tiny$4$};
\node[below=0.05cm of 7000] {\tiny$5$};

\path[draw]
(0)--(1)
(2)--(3)
(3)--(4000)--(5000)
(2)--(400)--(500)
(1)--(40)--(50)
(0)--(4)--(5)
(6)--(7)
(60)--(70)
(600)--(700)
(6000)--(7000)
(4)--(40)
(5)--(50)
(6)--(60)
(7)--(70)
(400)--(4000)
(500)--(5000)
(600)--(6000)
(700)--(7000);

\path[draw]
(1)--(2)
(40)--(400)
(50)--(500)
(60)--(600)
(70)--(700)
(5)--(6)
(50)--(60)
(500)--(600)
(5000)--(6000);
\end{tikzpicture}
\caption{The heap of $L(A_8,\omega_5)$.}
\label{fig:a}
\end{figure}
\end{proof}


\section{Concluding remarks}
\label{sec:conclude}

We discuss a few open problems concerning the partial order $\le_k$ in this
section. First, it would be interesting to have a conceptual, case-free proof
of Theorem \ref{thm:repthy}.  Secondly, apart from the minuscule poset setting,
it may be interesting to study the structure of posets of the form
$\cala_k(\Phi^+)$ where $\Phi^+$ is a root poset, i.e., the poset of positive
roots of a Weyl group $W$. Thirdly, if $W$ has rank $r$, then the so-called
Narayana numbers $\abs{\cala_k(\Phi^+)}$ are symmetric in the sense that
$\abs{\cala_k(\Phi^+)}=\abs{\cala_{r-k}(\Phi^+)}$ for any $0\le k\le r$ (see
\cite{DefantHopkins}), so it would also be interesting to know whether this
symmetry can be realized by a poset isomorphism (with respect to the orders
$\le_k$ and $\le_{r-k}$) between $\cala_k(\Phi^+)$ and $\cala_{r-k}(\Phi^+)$.

Bijections realizing the symmetry
$\abs{\cala_k(\Phi^+)}=\abs{\cala_{r-k}(\Phi^+)}$ have been studied before.  In
\cite[Conjecture 6.1]{Panyushev}, Panyushev conjectures that for all root
systems $\Phi$, there is a natural involution $*$ on $\cala(\Phi^+)$ satisfying
a certain list of properties, one of which is that it should restrict to
bijections between $\cala_k(\Phi^+)$ and $\cala_{n-1-k}(\Phi^+)$ for all $0\le
k\le r$. Panyushev also constructs (\cite[Section 4]{Panyushev}) such an
involution for the root poset $\Phi^+=\{\varepsilon_i-\varepsilon_j: 1\le i<j\le
n\}\subseteq \R^n$ of type $A_{n-1}\, (n\ge 2)$ as follows: write
$[i,j]=\varepsilon_i-\varepsilon_j$ for all $1\le i<j\le n$, and for each antichain
$A=\{[i_1,j_1],\cdots, [i_k,j_k]\}$ in $\cala_k(\Phi^+)$, define $A^*$ to the
unique antichain in $\cala_{n-1-k}$ consisting of elements $[i'_1,j'_1], \cdots,
[i'_{n-1-k}, j'_{n-1-k}]$ where \[ \{i'_1,\cdots,
i'_{n-1-k}\}=\{1,2,\cdots,n-1\}\setminus\{j_1-1,\cdots, j_k-1\},\]
\[\{j'_1,\cdots, j'_{n-1-k}\}=\{2,3,\cdots,n\}\setminus\{i_1+1,\cdots,i_k+1\} \]
as sets and \[ i'_1<\cdots< i'_{n-1-k},\quad j'_1<\cdots< j'_{n-1-k}.  \] Using
Proposition \ref{prop:basics} (ii) and the fact that $[i,j]\lessdot [l,m]$ in
$\Phi^+$ if and only if either $i=l, m=j+1$ or $l=i-1, j=m$, it is
straightforward to verify that if $A'\lessdot A$ for some antichain $A'\in
\cala_k(\Phi^+)$, then $A'^*\lessdot A^*$ in $\cala_{n-1-k}(\Phi^+)$. Since the
map $*$  is an involution, we deduce the following: 
\begin{proposition} 
    \label{prop:panyu} For the root poset $\Phi^+$ of type $A_{n-1}$,
      Panyushev's natural involution $*$ on the set $\cala(\Phi^+)$ restricts to
      poset isomorphisms between $\cala_k(\Phi^+)$ and $\cala_{n-1-k}(\Phi^+)$
      for all $1\le k\le n-1$.\qed
\end{proposition} 

More generally, Defant and Hopkins prove in \cite{DefantHopkins} that for root
systems of types $A, B, C$ and $D$, a so-called {rowvacuation} operator
satisfies Panyushev's desired properties and recovers the map $*$ in type $A$.
However, we note that while rowvacuation provides bijections between
$\cala_k(\Phi^+)$ and $\cala_{r-k}(\Phi^+)$, it does not give a poset
isomorphism between these posets in types $B,C$ and $D$. To see this, recall
that $\Phi^+$ is a ranked poset. Let $R$ be the rank of  $\Phi^+$ and let
$\Phi^+_i$ be the antichain in $\Phi^+$ consisting of all elements of rank $i$
for each $0\le i\le R$.  Then in types $B,C$ and $D$, both $\Phi^+_R$ and
$\Phi^+_{R-1}$ are singletons satisfying $\Phi^+_{R-1}\lessdot \Phi^{+}_{R}$ in
$\cala_1(\Phi^+)$. On the other hand, Proposition 2.9 of \cite{DefantHopkins}
implies that rowvacuation sends $\Phi^+_{R-1}$ and $\Phi^+_{R}$ to $\Phi^+_{2}$
and $\Phi^+_1$, respectively, yet $\Phi^+_2$ and $\Phi^+_1$ are elements of
$\cala_{r-1}(\Phi^+)$ that are not in a covering relation.



\section*{Acknowledgements}
We thank Darij Grinberg and Hugh Thomas for helpful discussions. We also thank
the anonymous referees for reading our paper carefully and suggesting many
improvements.

\bibliographystyle{alphaurl}
\bibliography{gx4.bib}

\end{document}